\documentclass[smallextended]{svjour3}

\RequirePackage{fix-cm}

\smartqed

\usepackage{amsfonts}
\usepackage{amsmath}
\usepackage{mathtools}
\usepackage{tikz}

\usetikzlibrary{math, calc}

\DeclareMathOperator{\tr}{tr}
\DeclareMathOperator{\re}{Re}
\DeclareMathOperator{\Span}{span}
\DeclareMathOperator{\id}{id}
\DeclareMathOperator{\eucl}{Eucl}
\DeclareMathOperator{\ext}{ext}

\newcommand{\dens}{{\mathcal{D}(n)}}
\newcommand{\densz}{{\mathcal{D}_0(n)}}
\newcommand{\densp}{{\mathcal{D}_\textbf{p}(n)}}
\newcommand{\herm}{\mathbb{H}(n)}
\newcommand{\pherm}{{\mathbb{P}(n)}}
\newcommand{\BW}[3]{\mathfrak{g}^{BW}_{#3}({#1},{#2})}
\newcommand{\BWh}{\mathfrak{g}^{BW}}
\newcommand{\sgl}{{\mathbb{S}_{\mathrm{GL}(n)}}}
\newcommand{\bwd}[1]{d^{BW}_{#1}}

\newcommand{\Euclh}{\bar{\mathfrak{g}}}
\newcommand{\Eucl}[3]{\bar{\mathfrak{g}}_{#3}({#1},{#2})}
\newcommand{\Fis}[3]{\mathfrak{g}^F_{#3}({#1},{#2})}
\newcommand{\Fish}{\mathfrak{g}^F}

\newcommand{\deucl}{\bar{d}}
\newcommand{\SLD}[3]{\mathfrak{g}^{SLD}_{#3}({#1},{#2})}
\newcommand{\SLDh}{\mathfrak{g}^{SLD}}
\newcommand{\MO}{\mathcal{M}_+(\Omega)}
\newcommand{\PO}{\mathcal{P}_+(\Omega)}
\newcommand{\Cnn}{\mathbb{C}^{n \times n}}
\newcommand{\cM}{\mathcal{M}}
\newcommand{\cN}{\mathcal{N}}

\title{Bures-Wasserstein geometry for positive-definite Hermitian matrices and their trace-one subset\thanks{The author wishes to thank Giovanni Pistone for supervising this research project and acknowledges the support of the Collegio Carlo Alberto and the Prins Bernhard Cultuurfonds. This work was presented at the Information Geometry seminar in Turin, December 2019.}}
\author{Jesse van Oostrum}

\institute{Jesse van Oostrum \at
              Hamburg University of Technology, Am Schwarzenberg-Campus 1, 21073 Hamburg, Germany\\
              \email{jesse.van@tuhh.de}           %  \\
}

\date{Received: date / Accepted: date}

\begin{document}

\maketitle

\begin{abstract}
    In his classical argument, Rao derives the Riemannian distance corresponding to the Fisher metric using a mapping between the space of positive measures and Euclidean space.  He obtains the Hellinger distance on the full space of measures and the Fisher distance on the subset of probability measures.  In order to highlight the interplay between Fisher theory and quantum information theory, we extend this construction to the space of positive-definite Hermitian matrices using Riemannian submersions and quotient manifolds. The analog of the Hellinger distance turns out to be the Bures-Wasserstein (BW) distance, a distance measure appearing in optimal transport, quantum information, and optimisation theory.  First we present an existing  derivation of the Riemannian metric and geodesics associated with this distance. Subsequently, we present a novel derivation of the Riemannian distance and geodesics for this metric on the subset of trace-one matrices, analogous to the Fisher distance for probability measures.

\end{abstract}

\section{Introduction}

In this paper we investigate the geometrical properties of the Bures-Wasserstein (BW) distance on the space of positive-definite Hermitian matrices, $\pherm$. For $\Sigma, T \in \pherm$, this distance function is given by:
\begin{equation} \label{BWdistance1}
  d_\pherm^{BW}(\Sigma,T) = \left[\tr(\Sigma) + \tr(T) - 2 \tr\left(\left(\Sigma^{1/2}T\Sigma^{1/2}\right)^{1/2}\right)\right]^{1/2}.
\end{equation}
This function appears in optimal transport as a distance measure on the space of mean-zero Gaussian densities, where it is called the Wasserstein distance. In quantum information theory, this is a distance measure between quantum states or density matrices, called the Bures distance. \\

The development of this subject started when Rao realised that the Fisher information defines a Riemannian metric on the space of probability measures \cite{rao}. 
He obtained this metric by finding an isometry between the positive orthant of the unit sphere, equipped with the Euclidean metric, and the probability simplex. Later, the study of the geometrical properties of the probability simplex was extended to the quantum realm with notable contributions of Nagaoka, Petz and Hasegawa. An overview of this field can be found in  \cite{hayashi} and \cite{bengtsson}. The distance measure defined in \eqref{BWdistance1} was introduced by Helstrom \cite{helstrom} and Bures \cite{bures} as a measure of similarity between quantum states. In \cite{uhlmann}, Uhlmann derives the geometrical properties of this distance measure.\\

In the context of optimal transport, this distance measure was derived to be the $L^2$-Wasserstein distance on the space of covariance matrices for mean zero Gaussian distributions \cite{olkin}. 
Its geometrical properties were first studied in this context by Takatsu \cite{takatsu}. It is interesting to note that the obtained results are equivalent to but independent of the work of Uhlmann fifteen years prior. See Section \ref{BWQ} for more on this. The Levi-Civita connection and exponential map are computed in \cite{pis1}. A further question at the end of \cite{pis1} on the restriction to trace-one matrices formed the start of the research for this paper.  In \cite{bhatia1}, Bhatia discusses both the quantum and the optimal transport interpretation of the distance and introduces the name Bures-Wasserstein distance. Furthermore, the approach of Takatsu is refined using facts on Riemannian submersions and quotient manifolds. The curvature and parallel transport for the BW metric are examined in \cite{massart2019curvature} and \cite{thanwerdas2021n}. More recently, the geometrical structure discussed in this paper is of interest in the field of optimisation. It turns out that for this choice of geometry the exponential and logarithmic map are cheap to evaluate, which makes it particularly suitable for numerical computations \cite{Massart20}. In the present paper, we adapt the quotient manifold construction from Bhatia to the subspace of trace-one positive-definite Hermitian matrices, which is of particular importance in quantum information.  We recover the Riemannian distance for the BW metric on this space, which is called the \emph{Bures-Wasserstein angle}, and its the geodesics. This distance measure and geodesics were first derived by Uhlmann in the context of quantum information \cite{uhlmann}. \\

 The main goal of this paper is to highlight the interplay between Fisher theory and quantum information theory. 
 This is done by showing the parallel between the classical derivation of the Hellinger and Fisher distance due to Rao, and the derivation of the BW distance and angle. For this, we present a novel rigorous derivation of the BW angle that emphasises the geometrical aspects of the BW metric and does not utilise concepts from quantum information. This derivation is an extension of recent derivations for general positive-definite matrices, as presented in \cite{bhatia1,Massart20}. We then show how our derivation can be translated into the language of quantum information. A particular source of confusion is the use of different representations of the tangent space, as is also noted in \cite{thanwerdas2021n}. We present clear and general definitions of the the $(m)$- and $(e)$-representation used in quantum information in the language of differential geometry and show how different quantum information metrics can be denoted in each representation. \\

The next section of the paper discusses preliminary facts on differential geometry and matrix identities and gives general definitions for the $(m)$- and $(e)$-representation, which coincide with existing definitions from both classical and quantum information geometry. Section \ref{cig} illustrates the use of these representations in the context of classical information geometry and describes the derivation of the Hellinger and Fisher distance due to Rao. In Section \ref{qig} the representations of the tangent space are worked out explicitly for the SLD Fisher metric and the Bogoliubov metric used in quantum information. The main derivations of the paper can be found in Section \ref{BWgeom}, where the geometrical structure of the BW distance is investigated first on $\pherm$ and subsequently on the the trace-one subset, $\dens$. The last section of the paper compares the geometrical structure obtained in the foregoing to similar results in the field. An overview of the notation used in the paper can be found on page \pageref{notation}.

\section{Preliminaries} \label{prelimi}

\subsection{Differential geometry}

Let $\mathcal{M}, \mathcal{N}$ be smooth manifolds and $F: \mathcal{M} \to \mathcal{N}$ a smooth map. We will denote the differential of $F$ at $p \in \mathcal{M}$ by: $dF_p: T_p \mathcal{M} \to T_{F(p)}\mathcal{N}$. For $g$ a Riemannian metric on $\mathcal{M}$, the length of a tangent vector $v \in T_p \mathcal{M}$ is given by: $||v||_g = (g(v,v))^{1/2}$, and the length of a curve $\gamma: [a,b] \to \mathcal{M}$ is given by:
\begin{equation}
  L(\gamma) = \int_a^b ||\gamma'(t)||_g dt
\end{equation}
where $\gamma'(t_0) = d\gamma_{t_0}(\frac{d}{dt}|_{t_0})$. The Riemannian distance between $p,q \in \mathcal{M}$ is defined to be:
\begin{equation} \label{riemdis}
  d_\mathcal{M}(p,q) = \inf \{ L(\gamma) \mid \gamma: [0,1] \to \mathcal{M}, \gamma(0) = p, \gamma(1) = q \}.
\end{equation}
See chapter 2 of \cite{lee2018} for details.

\begin{definition}
  Let V be a real or complex vector space. An \textit{affine subspace} of $V$ is a subset $A \subset V$ together with a vector subspace $\widetilde{V} \subset V$ such that:
  \begin{itemize}
    \item $\forall a,b \in A, \ \exists v \in \widetilde{V}$ such that $a+v = b$
    \item $\forall v \in \widetilde{V}$ and $a \in A$ we have:  $a+v \in A$
  \end{itemize}
\end{definition}

Now let $\mathcal{M}$ be an open convex subset of an affine subspace $A$ of $V$ with associated vector space $\widetilde{V}$ and  $p \in \mathcal{M}$ fixed. We have that the following map is a vector space isomorphism \cite{lee2012}:
\begin{align}
  \id_1: \widetilde{V} &\to T_p \mathcal{M} \\
  \tilde{v} &\mapsto v \\
  \id_1(\tilde{v})(f) &\coloneqq \frac{d}{dt} \bigg\rvert_{t=0} f(p + t\tilde{v}),
\end{align}
 where $f$ is any smooth function on $\mathcal{M}$. We can therefore identify every tangent vector in $T_p \mathcal{M}$ with an element of $\widetilde{V}$ through $\id_1$. Given a basis $(e_1, ..., e_m)$ for $V$, we define the Euclidean inner product on $V$ to be such that $\langle e_i, e_j \rangle^{\eucl} \coloneqq \delta_{ij}$. The \textit{Euclidean metric} $\Euclh$ on $\mathcal{M}$ is defined such that for $\tilde{v}, \tilde{w} \in \widetilde{V}$, we have $\Euclh_p(\id_1(\tilde{v}), \id_1(\tilde{w})) = \re \left(\langle \tilde{v}, \tilde{w} \rangle^{\eucl} \right)$. The Riemannian distance associated with $\Euclh$ is denoted $\bar{d}$.\\

\subsection{Representations of the tangent space}

Let $\mathbb{K}$ be either $\mathbb{R}$ or $\mathbb{C}$ and $(e_1, ..., e_m)$ a fixed basis for $V$. This basis induces a coordinate map $k: V \to \mathbb{K}^m$  such that for $k(v) = x$ we have $v = \sum_i x_i e_i $. We define the \textit{$(m)$-representation} of an element in $T_p \mathcal{M}$ to be the coordinate representation of its $\id_1$-associated element in $\widetilde{V}$. Or in symbols,
\begin{equation}
  (m):T_p\mathcal{M} \xrightarrow{\id_1} \widetilde{V} \xrightarrow{k} \mathbb{K}^m.
\end{equation}
Using the Riesz representation theorem we let $\id_2$ be the identification between $T_p \mathcal{M}$ with its dual $T^*_p \mathcal{M}$ through the Euclidean metric such that: $v \leftrightarrow v'(\cdot) = \Euclh_p(v,\cdot)$. We define the $(m^*)$-representation of an element of $T_p^*\mathcal{M}$ to be the $(m)$-representation of its $\id_2$-associated element in $T_p\mathcal{M}$. In symbols,
\begin{equation}
  (m^*): T_p^*\mathcal{M} \xrightarrow{\id_2} T_p \mathcal{M}\xrightarrow{(m)} \mathbb{K}^m.
\end{equation}
A general Riemannian metric $\mathfrak{g}$ on $\mathcal{M}$ gives a final identification, $\id_3(\mathfrak{g})$, between $T_p \mathcal{M}$ and $T^*_p \mathcal{M}$: $v \leftrightarrow v'(\cdot) = \mathfrak{g}_p(v,\cdot)$. Given a metric $\mathfrak{g}$, we define the \textit{$(e)$-representation} of an element of $T_p \mathcal{M}$ as the $(m^*)$-representation of its $\id_3(\mathfrak{g})$-associated element in $T^*_p \mathcal{M}$. In symbols,
\begin{equation}
  (e): T_p\mathcal{M} \xrightarrow{\id_3(\mathfrak{g})} T_p^* \mathcal{M}\xrightarrow{(m^*)} \mathbb{K}^m.
\end{equation}
Note that the definition of the $(m)$- and $(e)$-representation implies:
\begin{equation}
  \mathfrak{g}_p(v,w) = \re \left( \langle v^{(e)}, w^{(m)} \rangle \right),
\end{equation}
with on the right the standard inner product on $\mathbb{K}^m$.

\begin{remark}
  The definitions above are inspired by Chapter 2 of \cite{ay}. We will see in Sections \ref{cig} and \ref{qig} that they correspond to the standard definitions in classical and quantum information theory given in e.g. \cite{amari2007} and \cite{hayashi}.
\end{remark}

\subsection{Matrix identities}
Let $\mathrm{GL}(n)$ be the space of invertible complex matrices and $\mathrm{U}(n)$ the unitary matrices. Every $M \in \mathrm{GL}(n)$ can be written as $M = U \Sigma$ where $\Sigma = \left(M^*M\right)^{1/2} \in \pherm$ and $U = M \Sigma^{-1} \in \mathrm{U}(n)$. This is called the \textit{polar decomposition} of $M$ and $U$ is called the \textit{unitary polar factor}. In the following result, related to Uhlmann's theorem and the Procrustes problem, the unitary polar factor shows up in a maximisation problem.
\begin{theorem} \label{Uhlmann}
  Let $\Sigma, T \in \pherm$ and consider the following maximisation problem:
  \begin{equation}
    \sup_{V \in \mathrm{U}(n)} \re \tr\left(\Sigma V T\right).
  \end{equation}
  Then the supremum is attained for $V = U^*$, with $U$ the unitary polar factor of $T\Sigma$.
\end{theorem}
\begin{proof}
  See the proof of Theorem 1 in \cite{bhatia1}, or Theorem 9.2 in \cite{bengtsson}. \qed
\end{proof}
\begin{definition}
    The solution for $X$ to the \emph{Lyapunov equation}: $\Sigma X + X \Sigma = H$ with $\Sigma \in \pherm, H \in \herm$ will be denoted $\mathcal{L}_\Sigma(H)$. It turns out that this solution exists and is unique \cite{bhatia2}.
\end{definition}

\section{Classical information geometry} \label{cig}

In this section we will study $\mathcal{M}_+(\Omega)$ and $\mathcal{P}_+(\Omega)$, the space of stricly positive (resp. probability) measures on $\Omega = \{\omega_1, ..., \omega_n\}$. These spaces are open convex subsets of affine subspaces of the vector space of signed measures $\mathcal{S}(\Omega)$ and therefore we can use the definitions from the previous section. The canonical basis of $\mathcal{S}(\Omega)$ is given by the set of Dirac delta measures $(\delta_1, ..., \delta_n)$ such that $\delta_i(\omega_j) = \delta_{ij}$. This basis gives us the $(m)$-representation as described in the preliminaries. Now we let the metric on $\MO$ and $\PO$ be the Fisher information metric, given by:
\begin{equation}
  \Fis{a}{b}{\mu} = \sum_{i=1}^n \frac{a^{(m)}_i b^{(m)}_i}{\mu(\omega_i)},
\end{equation}
with $a,b$ in the tangent space of $\MO$ respectively $\PO$ at $\mu$. The $(e)$-representation is given by:
\begin{equation}
  a^{(e)}_i = \frac{a_i^{(m)}}{\mu(\omega_i)}.
\end{equation}
Another way of obtaining the $(e)$-representation for this case is by applying the $(m)$-representation to the pushforward of $a$ under the logarithm map. If $\log: \mathcal{M}_+(\Omega) \ni \mu \mapsto \log(\mu) \in \mathcal{S}(\Omega)$ such that $\log(\mu)(\omega) = \log(\mu(\omega)) \ \forall \omega \in \Omega$, then:
\begin{equation}
  a^{(e)} = \left(d\log_\mu(a)\right)^{(m)}.
\end{equation}
We will see however that this expression of the $(e)$-representation is not general enough for the quantum case.

\subsection{Hellinger distance, Fisher  metric and Fisher distance}\label{argumentrao}
In this section we derive the Riemannian metric corresponding to the Hellinger distance on $\mathcal{M}_+(\Omega)$. Then we find the Riemannian distance corresponding to the restriction of this metric to $\mathcal{P}_+(\Omega)$. These will turn out to be the Fisher metric and Fisher distance respectively. This derivation was first due to Rao \cite{rao} and can be seen as a special case of the derivation in Section \ref{BWgeom}. \\

The Hellinger distance is given by: 
\begin{equation}
    d^H(\mu,\nu) = \sqrt{\sum_{i=1}^{n} \left(\mu(\omega_i)^{1/2} - \nu(\omega_i)^{1/2}\right)^2}. 
\end{equation}
It can be obtained as the pushforward of the Euclidean distance, $\bar{d}$, under the square map:
\begin{equation} \label{isometry0}
  (\mathcal{M}_+(\Omega), \deucl) \ni \mu \mapsto \mu^2 \in (\mathcal{M}_+(\Omega), d^H),
\end{equation}
 where $\mu^2(\omega) \coloneqq \mu(\omega)^2 \ \ \forall \omega \in \Omega$. Extending the geometrical structure on the left from a distance function to the corresponding Riemannian Euclidean metric $\Euclh$, we have the following isometry\footnote{The isometries in this section are defined up to a constant factor.}:
\begin{equation} \label{isometry1}
  \left(\mathcal{M}_+(\Omega), \Euclh\right) \ni \mu \mapsto \mu^2 \in \left(\mathcal{M}_+(\Omega), \Fish\right).
\end{equation}
 From \eqref{isometry0} and \eqref{isometry1} it follows that the Hellinger distance is the geodesic distance for the Fisher metric in $\mathcal{M}_+(\Omega)$. \\

We now aim to find the Riemannian distance for the Fisher metric restricted to $\mathcal{P}_+(\Omega)$. It turns out that for this subset, the geodesic distance is no longer the Hellinger distance. In order to find the right geodesic distance, we can use the fact that the following restriction of \eqref{isometry1} remains an isometry:
\begin{equation}
  \left(\mathbb{S}_{\mathcal{M}_+(\Omega)}, \Euclh\right) \ni \mu \mapsto \mu^2 \in \left(\mathcal{P}_+(\Omega), \Fish\right).
\end{equation}
where $\mathbb{S}_{\mathcal{M}_+(\Omega)} \coloneqq \{\mu \in \mathcal{M}_+(\Omega): \sum_i \mu(\omega_i)^2 = 1\}$, the unit sphere in $\left(\mathcal{M}_+(\Omega), \Euclh \right)$ (see Figure \ref{fig:isometry}). We know that on this space, the geodesics are given by great circles and therefore we can also compute the Riemannian distance. Using the fact that this distance is carried over by the isometry, we obtain the Riemannian distance for the space of probability measures with the Fisher metric, called the Fisher distance. This is given by:
\begin{equation}
  d^F(p,q) = \arccos\left(\sum_{i=1}^{n} \big(p(\omega_i)q(\omega_i)\big)^{1/2}\right).
\end{equation}

\tikzmath{\scale = 1; \axislength = \scale * 1.3; \anglemu=70; \anglenu=25; \xmu = cos(\anglemu)^2; \ymu = sin(\anglemu)^2; \xnu = cos(\anglenu)^2; \ynu = sin(\anglenu)^2;} 

\begin{figure}
\begin{tikzpicture}[auto, scale=3, axis/.style={->, thin, dotted}, dot/.style={fill, circle, inner sep =0, minimum size = 1.4 mm}]

\coordinate (a) at (2.2,0);

% axes
\draw[axis] (0,0) -- (\axislength, 0) node [anchor = north] {$\delta_1$};
\draw[axis] (0,0) -- (0, \axislength) node [anchor = east] {$\delta_2$};
\draw[axis] (a) -- ($(a) + (\axislength, 0)$) node [anchor = north] {$\delta_1$};
\draw[axis] (a) -- ($(a) + (0, \axislength)$) node [anchor = east] {$\delta_2$};

% manifolds
\draw [very thin] (1,0) arc[start angle=0, end angle=90, radius=1];
\draw [very thin] ($(a) + (0,1)$) -- ($(a) + (1,0)$); 

% points
\node (muinv) at (\anglemu:1cm) [dot, label=above right:$\pi^{-1}(p)$]  {};
\node (nuinv) at (\anglenu:1cm) [dot, label=right:$\pi^{-1}(q)$]  {};
\node (mu) at ($(a) + (\xmu, \ymu)$) [dot, label=above right:$p$] {};
\node (nu) at ($(a) + (\xnu, \ynu)$) [dot, label=right:$q$] {};

\node at (1.25, .15) {$\mathbb{S}_{\mathcal{M}_+(\Omega)}$};
\node at ($(a) + (1.25, .15)$) {$\mathcal{P}_+(\Omega)$};
% distances
\draw[thick] (muinv)  -- node[swap] {$d^H(p,q)$} (nuinv);

\draw [thick] (\anglemu:1.0cm) arc[start angle=\anglemu, end angle=\anglenu, radius=1.0]  node[pos=.5] {$d^F(p,q)$};

% map
\draw[->] (1, 1.2) to [bend left = 20] node {$\pi$} ($(a) + (-.5, 1.2)$);

\end{tikzpicture}
\caption{A mapping ($\pi$) from the sphere ($\mathbb{S}_{\cM_+}(\Omega)$) to the simplex ($\mathcal{P}_+(\Omega)$) and an illustration of the relation between the Hellinger ($d^H$) and Fisher ($d^F$) distance}
\label{fig:isometry}
\end{figure}

\section{Quantum information geometry} \label{qig}
Let $\herm$ be the set of Hermitian matrices and $\dens$ be the subset of positive-definite Hermitian matrices with trace one. Within the context of Section \ref{prelimi}, we have $\mathcal{M} = \dens$, $V = \Cnn$ and $\widetilde{V} = \{H \in \herm : \tr(H) = 0  \}$. The basis vectors for $\Cnn$ are simply given by $(A_{11}, A_{12}, ..., A_{nn})$, where the $ij$-th entry of $A_{ij}$ is one and the rest zero.  From this, we get the $(m)$-representation for $T_\rho \dens$. For the submanifold of diagonal matrices (probability measures) Chentsov showed that the Fisher metric is the unique metric satisfying certain (statistically) natural conditions on the metric \cite{chentsov}. Petz proved that for $\dens$, this uniqueness no longer exists \cite{petz}. One of the suggested generelisations is the symmetrised logarithmic derivative (SLD) Fisher metric, see e.g. \cite{amari2007,hayashi}. For $H, K \in T_\rho \dens$ this Hermitian metric is given explicitly by:
\begin{equation}
  \SLD{H}{K}{\rho} = 2 \tr\left(\mathcal{L}_\rho\left(H^{(m)}\right)K^{(m)}\right).
\end{equation}
We will derive in Section \ref{BWgeom} that the Riemannian metric corresponding to the BW distance on $\pherm$ is given by:
\begin{equation}
  \BW{H}{K}{\Sigma} = \frac{1}{2} \re \tr\left(\mathcal{L}_\Sigma\left(H^{(m)}\right) K^{(m)}\right).
\end{equation}
Furthermore we will prove that the Riemannian distance on $\dens$ for this metric is given by:
\begin{equation} \label{BWDdens}
  d^{BW}_\dens(\rho_1, \rho_2) = \arccos\left(\re\tr\left(\left(\rho_2^{1/2}\rho_1\rho_2^{1/2}\right)^{1/2}\right)\right).
\end{equation}
Because the real parts of $\SLDh$ and $\BWh$ are equal on $\dens$, we can conclude that $d^{BW}_\dens$ is the distance function for $\SLDh$ (up to a constant factor).

\subsubsection*{\textit{(e)-representations in quantum information geometry}}
The SLD Fisher metric for $H, K \in T_\rho \dens$ is given up to a constant by:
\begin{equation}
  \SLD{H}{K}{\rho} = \tr\left(\mathcal{L}_\rho\left(H^{(m)}\right)K^{(m)}\right).
\end{equation}
From the preliminaries it follows that for this choice of metric we have the following relation between the $(m)$- and $(e)$-representation:
\begin{align}
  H^{(e)} &= \mathcal{L}_\rho\left(H^{(m)}\right), \\
  H^{(m)} &= H^{(e)}\rho + \rho H^{(e)}.
\end{align}
Expressing the SLD Fisher metric in terms of the $(e)$-representation therefore gives the potentially more familiar expression:
\begin{equation}
  \SLD{H}{K}{\rho} = \tr\left(H^{(e)} \left(K^{(e)}\rho + \rho K^{(e)}\right) \right).
\end{equation}
Another common metric is the Bogoliubov metric. In the $(m)$-representation this is given by:
\begin{equation}
  \mathfrak{g}^{Bo}_\rho(H,K) = \tr\left( (d\log_\rho(H))^{(m)} K^{(m)}\right),
\end{equation}
where $\log$ is the inverse of the matrix exponential $\exp(\Sigma) = \sum_{k=0}^\infty \frac{1}{k} \Sigma^k$. The relation between the $(m)$- and $(e)$-representation is given by:
\begin{align}
  H^{(e)} &= (d\log_\rho(H))^{(m)}, \\
  H^{(m)} &= \int_0^1 \rho^\lambda H^{(e)} \rho^{1-\lambda} d\lambda.
\end{align}
The Bogoliubov metric in $(e)$-representation is therefore given by:
\begin{equation}
  \mathfrak{g}^{Bo}_\rho(H,K) = \tr\left( H^{(e)} \int_0^1 \rho^\lambda K^{(e)} \rho^{1-\lambda} d\lambda\right).
\end{equation}

\begin{remark}
  In the rest of the paper we will exclusively and implicitely use the $(m)$-representation for the elements of the tangent bundle of $\pherm$ and $\dens$.
\end{remark}

\section{Bures-Wasserstein Geometry} \label{BWgeom}

In this section, we explore the geometry induced by the Bures-Wasserstein distance. We start by  recovering the metric and geodesics corresponding to the BW distance on $\pherm$. For this, we follow the exposition in \cite{bhatia1} and provide extra justification for some of the results.\footnote{This derivation is presented here because it lays the foundation for the results derived in the second part of the section and it allows for a better comparison with the classical argument of Rao.} Subsequently, we restrict the obtained metric to $\dens$ and derive the corresponding distance function and geodesics. The flow of the argument is analogous to Section \ref{argumentrao}, where we start from the Hellinger distance on $\mathcal{M}_+(\Omega)$, derive the Fisher metric, and subsequently find the Riemannian distance and geodesics for this metric restricted to the submanifold $\mathcal{P}_+(\Omega)$. We start by discussing some general results from Riemannian geometry.\\

Let $(\mathcal{M},g)$ and $(\mathcal{N},h)$ be Riemannian manifolds and $\pi: (\mathcal{M},g) \to (\mathcal{N},h)$ a smooth submersion. We can make the following orthogonal decomposition of the tangent space at $p \in \mathcal{M}$:
\begin{equation}
  T_p \mathcal{M} = \mathcal{V}(\pi,p) \oplus \mathcal{H}(\pi,p,g),
\end{equation}
where $\mathcal{V}(\pi,p)$ is the kernel of $d\pi_p$ and $\mathcal{H}(\pi,p,g)$ is its orthogonal complement with respect to the metric at $p$. We will refer to these subspaces and their elements as \textit{vertical} and \textit{horizontal} respectively. A curve $\gamma$ in $\mathcal{M}$ is said to be horizontal if $\gamma'(t)$ is horizontal for all $t$. We say that a submersion $\pi$ is  \textit{Riemannian} if for all $p \in \mathcal{M}$ and $v,w \in \mathcal{H}(\pi,p,g)$ the following holds:
\begin{equation} \label{Riemsub}
  g_p(v,w) = h_{\pi(p)}(d\pi_p v, d\pi_p w).
\end{equation}
That is, the map $d\pi_p|_{\mathcal{H}(\pi,p,g)}$ is an isometric isomorphism.\\

\begin{theorem} \label{theorem2}
    Let $\pi: (\mathcal{M},g) \to (\mathcal{N},h)$ be a Riemannian submersion. For every horizontal curve $\gamma$ in $\cM$ we have:
    \begin{itemize}
        \item $L(\gamma) = L(\pi \circ \gamma)$.
    \end{itemize}
    For every geodesic $\gamma$ in $\mathcal{M}$ such that $\gamma'(0)$ is horizontal we have:
    \begin{itemize}
      \item $\gamma'(t)$ is horizontal for all t,
      \item $\pi \circ \gamma$ is a geodesic in $\mathcal{N}$.
    \end{itemize}
  \end{theorem}
  \begin{proof}
    Lemma 26.11 in \cite{michor2008topics} and Proposition 2.109 in \cite{gallot1990riemannian}. \qed
  \end{proof}

\begin{theorem} \label{theorem1}
  Suppose $(\mathcal{M},g)$ is a Riemannian manifold and $G$ a compact Lie group of isometries of $(\mathcal{M},g)$ acting freely and smoothly on $\mathcal{M}$. Then the orbit space $\mathcal{N} = \mathcal{M}/G$ has a unique smooth manifold structure and there exists a unique Riemannian metric $h$ such that the quotient map $\pi : (\mathcal{M},g) \to (\mathcal{N},h)$ is a Riemannian submersion. Furthermore the Riemannian distance $d_\cN$ is given by quotient distance\footnote{The distance of a quotient metric space is formally defined as: $\inf\{d(p_1, q_1) + ... + d(p_n,q_n)\}$ with the infimum over all finite sequences $(p_1, ..., p_n), (q_1,...,q_n)$ such that $p_1 \in \pi^{-1}(p), q_n \in \pi^{-1}(q)$. Proposition 4.1 in \cite{lang2006length} shows that in our setting this distance reduces to \eqref{quotdis}.}:
  \begin{equation} \label{quotdis}
      d'_\cN (p,q) \coloneqq \inf_{\substack{{\tilde{p} \in \pi^{-1}(p)} \\ {\tilde{q} \in \pi^{-1}(q)}}} d_\mathcal{M}(\tilde{p},\tilde{q}).
  \end{equation}
\end{theorem}
\begin{proof}
  The first part of the theorem is a standard result and can be found in e.g. Corollary 2.29 of \cite{lee2018}. We show the second part explicitly. Recall from equation \eqref{riemdis} in the preliminaries that the Riemannian distance function on $\mathcal{N}$ is given by:
  \begin{equation}
      d_\mathcal{N}(p,q) = \inf\{L(\gamma) \mid \gamma:[0,1] \to \mathcal{N}, \gamma(0) = p, \gamma(1)=q \}.
  \end{equation}
  From \cite{michor2008topics} we know that for the type of Riemannian submersion described in the theorem, we can find for every curve $\gamma$ on the RHS, a horizontal curve $\tilde\gamma$ in $\mathcal{M}$ such that $\pi \circ \tilde\gamma = \gamma$. So in particular we have $\tilde\gamma(0) \in \pi^{-1}(p)$ and $\tilde\gamma(1) \in \pi^{-1}(p)$. From Theorem \ref{theorem2} it follows that $L(\tilde\gamma) = L(\gamma)$. Therefore we have $d_\mathcal{N}(p,q) \geq d_\mathcal{N}'(p,q)$. \\

  \noindent For the reverse inequality, we note that for every curve $\tilde\gamma$ in $\mathcal{M}$ we have:
  \begin{align}
    L(\tilde\gamma) &= \int_0^1 ||\tilde\gamma'(t)||_g \ dt \\
      &\geq \int_0^1 || d\pi_{\tilde\gamma(t)}(\tilde\gamma'(t))||_h \ dt \\
      &= \int_0^1 || (\pi \circ \tilde\gamma)'(t))||_h \ dt \\
      &= L(\pi \circ  \tilde\gamma),
  \end{align}
  where in the second line we use that $\pi$ is a Riemannian submersion. From this it follows immediately that $d_\mathcal{N}(p,q) \leq d_\mathcal{N}'(p,q)$. \qed
\end{proof}

\subsection{Geometry on the space $\pherm$}

\subsubsection*{\textit{Riemannian metric and distance function}}

In this section we will prove the following theorem.
\begin{theorem} \label{BWM}
The Bures-Wasserstein distance on $\pherm$ given in \eqref{BWdistance1} is a Riemannian distance. The corresponding metric is given by:
\begin{equation} \label{BWmetric}
  \BW{H}{K}{\Sigma} = \re \tr(\mathcal{L}_\Sigma(H) \Sigma \mathcal{L}_\Sigma(K)) = \frac{1}{2} \re \tr(\mathcal{L}_\Sigma(H) K),
\end{equation}
with $H,K \in T_\Sigma \pherm \cong \herm$.
\end{theorem}

In order to prove Theorem \ref{BWM} we need some preliminary results. The final proof can be found on page \pageref{proof1}. \\

From Theorem \ref{theorem1} we know that there exists a metric $\tilde{h}$ such that the quotient map $(\mathrm{GL}(n),\Euclh) \to (\mathrm{GL}(n)/\mathrm{U}(n), \tilde{h})$ is a Riemannian submersion. We make the following identification: $\mathrm{GL}(n)/\mathrm{U}(n) \ni M \cdot \mathrm{U}(n) \leftrightarrow MM^* \in \pherm$. This gives us the following map:
\begin{align} \label{pi}
  \pi: \left(\mathrm{GL}(n),\Euclh \right) &\to (\pherm, h) \\
  M &\mapsto MM^*.
\end{align}
 We will derive the prelimary results in the following order. First, we find the horizontal and vertical subspaces for $\pi$ (Proposition \ref{subspaces}), which we use to show that $h$ is given by \eqref{BWmetric} (Proposition \ref{h=BW}). Next, we show that the BW distance on $\pherm$ is equal to the quotient distance given in \eqref{quotdis} for $\pi$ (Proposition \ref{distance1}). Then we can use Theorem \ref{theorem1} to conclude that the BW distance is actually the Riemannian distance for $\BWh$.

\begin{proposition} \label{subspaces}
  Let $\herm$ and $\mathbb{H}^\perp(n)$ be the set of Hermitian and skew-Hermitian matrices, respectively. The vertical and horizontal space of $\pi$ at $M \in (\mathrm{GL}(n), \Euclh)$ are given by:
  \begin{align}
    \mathcal{V}(\pi, M) &= \big\{ K \left(M^{-1}\right)^* : K \in \mathbb{H}^\perp(n) \big\}\label{verti}\\
    \mathcal{H}(\pi, M, \Euclh ) &= \big\{ H M : H \in \herm \big\}. \label{hori}
  \end{align}
\end{proposition}
\begin{proof}
  We have:
  \begin{equation}
    d\pi_M(A) = AM^* + M A^*.
  \end{equation}
  Therefore $d\pi_M(A) = 0 \iff A \in \mathcal{V}(\pi,M)$. Furthermore, we have that $\Euclh_M\left(K\left(M^{-1}\right)^*,A\right) = \re \tr\left(K\left(M^{-1}\right)^*A^*\right)= 0 \ \forall K \in \mathbb{H}^\perp(n) \iff A \in \mathcal{H}(\pi, M, \Euclh )$. \qed
\end{proof}

\begin{proposition} \label{h=BW}
  The metric $h$ on $\pherm$ is given by \eqref{BWmetric}.
\end{proposition}
\begin{proof}
  Because $\pi$ is a Riemannian submersion, we know from \eqref{Riemsub} that for $A,B \in \mathcal{H}(\pi, M, \Euclh )$, $h$ needs to satisfy:
  \begin{equation}
    \Eucl{A}{B}{M} = h_{MM^*}(d\pi_M A, d\pi_M B). \\
  \end{equation}
  Working this out gives:
  \begin{equation}
    \re \tr(AB^*)= h_{MM^*}(MA^* + AM^*, MB^* + BM^*). \label{inner1}
  \end{equation}
  Now we plug in $A = \tilde{H} M, B = \tilde{K} M$ for $\tilde{H}, \tilde{K} \in \herm$. Then \eqref{inner1} becomes:
  \begin{equation}
    \re \tr(\tilde{H}MM^*\tilde{K})= h_{MM^*}(MM^*\tilde{H} + \tilde{H}MM^*, MM^*\tilde{K} + \tilde{K}MM^*).
  \end{equation}
  If we set $M= \Sigma^{1/2}$ and $\tilde{H} = \mathcal{L}_\Sigma(H), \tilde{K} = \mathcal{L}_\Sigma (K)$, we get for general $\Sigma \in \pherm$ and $H,K \in \herm$:
  \begin{equation}
    h_\Sigma(H,K) = \re \tr(\mathcal{L}_\Sigma(H) \Sigma \mathcal{L}_\Sigma(K)).
  \end{equation}
  Using the properties of the trace we have:
  \begin{equation}
    \re \tr(\mathcal{L}_\Sigma(H) \Sigma \mathcal{L}_\Sigma(K)) = \re \tr(\mathcal{L}_\Sigma(K) \Sigma \mathcal{L}_\Sigma(H)) = \re \tr(\mathcal{L}_\Sigma(H) \mathcal{L}_\Sigma(K) \Sigma ).
  \end{equation}
  Adding the first and last expression gives:
  \begin{equation}
    2 h_\Sigma(H,K) = \tr\left[\mathcal{L}_\Sigma(H) \Big(\Sigma \mathcal{L}_\Sigma(K) + \mathcal{L}_\Sigma(K) \Sigma\Big)\right] = \re \tr(\mathcal{L}_\Sigma(H) K).
  \end{equation}
  Dividing both sides by two gives the final result. \qed
\end{proof}

In order to show that the BW distance on $\pherm$ is equal to the quotient distance of $\pi$, we first have to investigate the distance on $(\mathrm{GL}(n),\Euclh)$. We know that on $(\mathbb{C}^{n \times n}, \Euclh)$ the distance is given by: $\bar d_{\mathbb{C}^{n \times n}}(A,B) = ||A - B||_2 = \left[\tr\left((A-B)(A-B)^* \right)\right]^{1/2}$. Because $\mathrm{GL}(n) \subset \mathbb{C}^{n \times n}$ we have $ \bar d_{\mathrm{GL}(n)} \geq \bar d_{\mathbb{C}^{n \times n}}$. However we can show, using the following lemmata, that for some choices of $A$ and $B$ the curve $\gamma(t) = (1-t)A + tB$ stays in $\mathrm{GL}(n)$ and thus the two distances are equal.

\begin{lemma} \label{prop1}
  For $\Sigma, T \in \pherm$ and $U = T\Sigma(\Sigma T^2 \Sigma)^{-1/2} $, the unitary polar factor of $T\Sigma$, we have that $TU\Sigma^{-1} \in \pherm$.
\end{lemma}
\begin{proof}
  Substituting $U$ gives:
  \begin{align}
    TU\Sigma^{-1} &= T^2\Sigma(\Sigma T^2 \Sigma)^{-1/2}  \Sigma^{-1}\\
    &= \Sigma^{-1} \Sigma T^2\Sigma(\Sigma T^2 \Sigma)^{-1/2}  \Sigma^{-1}\\
    &= \Sigma^{-1} (\Sigma T^2 \Sigma)^{1/2}  \Sigma^{-1}.
  \end{align} 
  It is easy to see that this is Hermitian. Positive definiteness follows from the fact that $(T\Sigma^2T)^{-1/2}$ is positive-definite. \qed
\end{proof}

\begin{lemma} \label{geod1}
  Let $\Sigma, T \in \pherm$ and $U,V$ such that $UV^* = T\Sigma(\Sigma T^2 \Sigma)^{-1/2} $, the unitary polar factor of $T\Sigma$. Then, $\gamma(t) = (1-t)\Sigma V + t T U$ is in $\mathrm{GL}(n)$ for $t \in [0,1]$.
\end{lemma}
\begin{proof}
  We can write:
  \begin{equation}
    \gamma(t) = \left((1-t)I + tTUV^*\Sigma^{-1}\right) \Sigma V.
  \end{equation}
  By the previous lemma know that $TUV^*\Sigma^{-1}$ is positive-definite. Therefore we have that $\left( (1-t)I + tTUV^*\Sigma^{-1} \right)$ is positive-definite for $t \in [0,1]$ and thus in $\mathrm{GL}(n)$. Since $\mathrm{GL}(n)$ is closed under multiplication we have that $\gamma(t) \in \mathrm{GL}(n)$. \qed
\end{proof}

Now we are in position to study the quotient distance \eqref{quotdis} for $\pi$. We have that $(\mathcal{M},g) = (\mathrm{GL}(n), \Euclh)$ and $\mathcal{N} = \pherm$.  Plugging this in gives the following distance function:
\begin{align}
d'_\pherm(\Sigma_1,\Sigma_2) &= \inf\{ \deucl_{\mathrm{GL}(n)}(M_1, M_2): M_i \in \pi^{-1}(\Sigma_i)\} \label{quotdis2}\\
&= \inf_{U,V \in \mathrm{U}(n)} \deucl_{\mathrm{GL}(n)}(\Sigma_1^{1/2} V, \Sigma_2^{1/2} U).
\end{align}

\begin{proposition} \label{distance1}
  The BW distance on $\pherm$ is equal to the quotient distance $d'_\pherm$. That is,
  \begin{multline}
    \left[tr(\Sigma_1) + \tr(\Sigma_2) - 2 \tr\left(\left(\Sigma_1^{1/2}\Sigma_2\Sigma_1^{1/2}\right)^{1/2}\right)\right]^{1/2} = \\ \inf_{U,V \in \mathrm{U}(n)} \deucl_{\mathrm{GL}(n)}(\Sigma_1^{1/2} V, \Sigma_2^{1/2} U).
  \end{multline}
  Moreover, the infimum on the right is attained when $UV^*$ is the unitary polar factor of $\Sigma_2^{1/2}\Sigma_1^{1/2}$, given by $\Sigma_2^{1/2}\Sigma_1^{1/2} \left( \Sigma_1^{1/2} \Sigma_2 \Sigma_1^{1/2} \right)^{-1/2}$.
\end{proposition}
\begin{proof}
  From the discussion above we know:
  \begin{align}
    d'^2_\pherm(\Sigma_1,\Sigma_2) &= \inf_{U,V \in \mathrm{U}(n)} \deucl_{\mathrm{GL}(n)}^2(\Sigma_1^{1/2} V, \Sigma_2^{1/2} U) \\
    &\geq \inf_{U,V \in \mathrm{U}(n)} \deucl_{\Cnn}^2(\Sigma_1^{1/2} V, \Sigma_2^{1/2} U) \label{distances}\\
    &= \inf_{U,V \in \mathrm{U}(n)} ||\Sigma_1^{1/2}V - \Sigma_2^{1/2}U||_2^2. \\
    &= \inf_{U,V \in \mathrm{U}(n)} \tr\left(\left(\Sigma_1^{1/2}V - \Sigma_2^{1/2}U\right)\left(\Sigma_1^{1/2}V - \Sigma_2^{1/2}U\right)^*\right)\\
    &= \tr(\Sigma_1) + \tr(\Sigma_2) - 2 \sup_{U,V \in \mathrm{U}(n)} \re \tr\left(\Sigma_1^{1/2}VU^*\Sigma_2^{1/2}\right).
  \end{align}
  We saw in Theorem \ref{Uhlmann} of the preliminaries that the supremum on the right is obtained for $UV^*$ as in the proposition. Moreover, by Lemma \ref{geod1} we have that for this choice $(1-t)\Sigma_1^{1/2}V + t\Sigma_2^{1/2}U$ stays in $\mathrm{GL}(n)$ and thus we have $\deucl_{\mathrm{GL}(n)}(\Sigma_1^{1/2} V, \Sigma_2^{1/2} U) = \deucl_{\mathbb{C}^{n \times n}}(\Sigma_1^{1/2} V, \Sigma_2^{1/2} U)$. Therefore we get equality in \eqref{distances} and conclude:
  \begin{align}
    d'_\pherm(\Sigma_1,\Sigma_2) &= \left[\tr(\Sigma_1) + \tr(\Sigma_2) - 2 \tr\left(\left(\Sigma_1^{1/2}\Sigma_2\Sigma_1^{1/2}\right)^{1/2}\right)\right]^{1/2}\\
    &= \bwd{\pherm}(\Sigma_1,\Sigma_2).
  \end{align} \qed
\end{proof}

\begin{proof} \textit{(of Theorem \ref{BWM})} \label{proof1}
By Theorem \ref{theorem1} we know that there exists a unique metric $h$ such that $\pi$ as defined in equation \eqref{pi} is a Riemannian submersion. In Proposition \ref{h=BW} we saw that this metric is given by $\BWh$, as defined in \eqref{BWmetric} in the statement of the theorem. From Theorem \ref{theorem1} we know that the Riemannian distance for this metric is given by the quotient distance for $\pi$, \eqref{quotdis2}. In Proposition \ref{distance1} we saw that this distance is equal to the BW-distance on $\pherm$. This is what we set out to proof. \qed
\end{proof}

\subsubsection*{\textit{Geodesics}}

In order to find a geodesic between $\Sigma_1$ and $\Sigma_2$ in $\pherm$, according to Theorem \ref{theorem2}, we need to find a geodesic $\gamma$ in $\mathrm{GL}(n)$ between points in $\pi^{-1}(\Sigma_1)$ and $\pi^{-1}(\Sigma_2)$ such that $\gamma'(0)$ is horizontal.

\begin{theorem} \label{geop}
  A geodesic between $\Sigma_1$ and $\Sigma_2$ in $\left(\pherm, \BWh \right)$ is given by $\pi \circ \gamma$, where
  \begin{equation}
    \gamma(t) = (1-t)\Sigma_1^{1/2}V + t \Sigma_2^{1/2}U
  \end{equation}
  and $UV^*$ is again the unitary polar factor of $\Sigma_2^{1/2}\Sigma_1^{1/2}$.
\end{theorem}
\begin{proof}
  It is clear that $\gamma$ is a geodesic in  $\left(\mathbb{C}^{n\times n}, \Euclh\right)$. We saw in Lemma \ref{geod1} that $\gamma$ stays in  $\mathrm{GL}(n)$. It remains to show $\gamma'(0) \in \mathcal{H}(\pi, \Sigma_1^{1/2}, \Euclh)$. We have:
  \begin{align}
    \gamma'(0) &= \Sigma_2^{1/2}U - \Sigma_1^{1/2}V\\
    &= \left(\Sigma_2^{1/2}UV^*\Sigma_1^{-1/2} - I\right) \Sigma_1^{1/2}V.
  \end{align}
  By Lemma \ref{prop1} we have that $\Sigma_2^{1/2}UV^*\Sigma_1^{-1/2}$ is Hermitian and thus the same holds for $\Sigma_2^{1/2}UV^*\Sigma_1^{-1/2} - I$. Using Proposition \ref{subspaces} we conclude $\gamma'(0)$ is horizontal. The statement of the theorem now follows from Theorem \ref{theorem2}. \qed
\end{proof}

\subsection{Geometry on the space $\dens$}

\subsubsection*{\textit{Inner product and distance function}}

Let us denote the unit sphere in  $\left(\mathbb{C}^{n\times n}, \Euclh\right)$ by:
\begin{equation}
  \mathbb{S}_{\mathbb{C}^{n\times n}} \coloneqq \{A \in \mathbb{C}^{n\times n} : \tr(AA^*) = 1\}
\end{equation}
and $\sgl \coloneqq \mathbb{S}_{\mathbb{C}^{n\times n}} \cap \mathrm{GL}(n)$, its restriction to $\mathrm{GL}(n)$. Note that:
\begin{equation}
  \pi^{-1}(\dens) = \sgl.
\end{equation}
We now apply Theorem \ref{theorem1} to this submanifold of $\mathrm{GL}(n)$. We choose the metric $g$ to be the restriction of $\Euclh$ to $\sgl$ and the Lie group $G$ again $\mathrm{U}(n)$. Since both the metric and the quotient map are just restrictions of the ones in Theorem \ref{BWM}, the resulting metric on $\sgl / \mathrm{U}(n) \cong \dens$ will also be the restriction of $\BWh$. From Theorem \ref{theorem1} we therefore know that the Riemannian distance function on $\dens$ corresponding to this restricted metric is given by the quotient distance defined in \eqref{quotdis}. Just as in the classical case (Section \ref{cig}) where the Fisher distance on $\PO$ is different from the Hellinger distance, it will turn out that the Riemannian distance on $\dens$ is different from the BW distance on $\pherm$. In order to compute the distance on $\dens$, we first investigate the geometry on $\sgl$. \\

Geodesics on a Euclidean sphere are obtained by intersecting the sphere with (hyper)planes through the origin.  If $M$ and $N$ are two non-antipodal point on $\mathbb{S}_{\mathbb{C}^{n\times n}}$ we can obtain the unnormalised geodesic by projecting the geodesic in $\mathbb{C}^{n\times n}$ onto  $\mathbb{S}_{\mathbb{C}^{n\times n}}$. More specifically, if $\gamma(t)= (1-t)M + (t)N$ is the geodesic in $\mathbb{C}^{n\times n}$, then
\begin{equation}
  \tilde\gamma(t) = \frac{\gamma(t)}{||\gamma(t)||_2}
\end{equation}
is the unnormalised geodesic in $\mathbb{S}_{\mathbb{C}^{n\times n}}$. Moreover, we have that $\gamma(t) \in \mathrm{GL}(n) \implies \tilde\gamma(t) \in \mathrm{GL}(n)$ since they are scalar multiples of each other. The Euclidean distance on $\mathbb{S}_{\mathbb{C}^{n\times n}}$ is given by:
\begin{equation}
  \bar{d}_{\mathbb{S}_{\mathbb{C}^{n\times n}}}(M,N)= \arccos(\re \tr (MN^*)).
\end{equation}
Just as before, we have that in general $\bar{d}_{\mathbb{S}_{\mathbb{C}^{n\times n}}} \leq \bar{d}_\sgl$, but when $\tilde{\gamma}$ stays in $\mathrm{GL}(n)$, we have that the two distances are equal. We are now in position to deduce the Riemannian distance function on $\dens$.

\begin{theorem}
  On $\dens$, the Riemannian distance for the Bures-Wasserstein metric is given by:
  \begin{equation}
    d^{BW}_\dens(\rho_1, \rho_2) = \arccos\left(\tr\left(\left(\rho_2^{1/2}\rho_1\rho_2^{1/2}\right)^{1/2}\right)\right).
  \end{equation}
  We will refer to this as the \emph{Bures-Wasserstein angle}
\end{theorem}
\begin{proof}
  From the definition of the quotient distance we have:
\begin{align}
  \bwd{\dens}(\rho_1,\rho_2) &= \inf_{U,V \in \mathrm{U}(n)} \bar{d}_\sgl\left(\rho_1^{1/2}V, \rho_2^{1/2}U\right)\\
  &\geq \inf_{U,V \in \mathrm{U}(n)} \bar{d}_{\mathbb{S}_{\mathbb{C}^{n\times n}}} \left(\rho_1^{1/2}V, \rho_2^{1/2}U\right) \label{ineq2} \\
  &= \inf_{U,V \in \mathrm{U}(n)} \arccos\left(\re \tr \left(\rho_1^{1/2}VU^*\rho_2^{1/2}\right)\right).
\end{align}
As before, the infimum is attained when $UV^*$ is the unitary polar factor of $ \rho_2^{1/2}\rho_1^{1/2}$. From Lemma \ref{geod1} and the above discussion we know that for this choice of $U$ and $V$ we have $(1-t)\rho_1^{1/2}V + t\rho_2^{1/2}U \in \mathrm{GL}(n)$. Therefore we have equality in \eqref{ineq2} and conclude the proof. \qed
\end{proof}

\subsubsection*{\textit{Geodesics}}
Analogous to Theorem \ref{geop}, we get the following result for the geodesics in $\dens$:
\begin{theorem}
  An unnormalised geodesic between $\rho_1$ and $\rho_2$ in $(\dens, \BWh)$ is given by $\pi \circ \tilde\gamma$, where
  \begin{equation}
    \tilde\gamma(t) = \frac{\gamma(t)}{||\gamma(t)||_2},
  \end{equation}
  with
  \begin{equation}
    \gamma(t) = (1-t)\rho_1^{1/2}V + t \rho_2^{1/2}U
  \end{equation}
  and $UV^*$ is the unitary polar factor of $\rho_2^{1/2}\rho_1^{1/2}$.
\end{theorem}
\begin{proof}
  Start by noting that the vertical and horizontal spaces are simply the restrictions of \eqref{verti} and \eqref{hori} to $T_{\rho^{1/2}}\sgl$. We saw above that $\tilde\gamma(t) \in \mathrm{GL}(n)$ and therefore by Theorem \ref{theorem2} it is enough to show $\tilde\gamma'(0)$ is horizontal. The direction of $\tilde\gamma'(0)$ can be obtained by projecting $\gamma'(0) \in T_{\rho_1^{1/2}V}\mathrm{GL}(n)$ onto the subspace $T_{\rho_1^{1/2}V}\sgl$. Since $\Span\{\rho_1^{1/2}V\}$ is the orthogonal complement of $T_{\rho_1^{1/2}V}\sgl$ within $T_{\rho_1^{1/2}V}\mathrm{GL}(n)$, this projection is given by:
  \begin{equation}
    \tilde\gamma'(0) \propto \gamma'(0) - \re \tr\left(\rho_1^{1/2}V\gamma'(0)^*\right)\rho_1^{1/2}V.
  \end{equation}
  From the proof of Theorem \ref{geop} it follows that $\gamma'(0)$ is horizontal in $T_{\rho_1^{1/2}V}\mathrm{GL}(n)$. Since $\rho_1^{1/2}V$, viewed as a tangent vector, is also horizontal in $T_{\rho_1^{1/2}V}\mathrm{GL}(n)$, it follows that $\tilde\gamma'(0)$ is horizontal in $T_{\rho_1^{1/2}V}\sgl$. \qed
\end{proof}

\section{BW geometry and quantum information} \label{BWQ}

In quantum information theory, a quantum state is represented  by a trace-one positive semi-definite matrix $\rho$ called a density operator. We will denote the set of these matrices by $\densz$. From the spectral theorem it follows that any $\rho \in \densz$ can be written as $\rho = U D U^*$ with $U \in \mathrm{U}(n)$ and $D$ a real non-negative diagonal matrix with trace one. Since this latter matrix can interpreted as a probability distribution on $\{1,2,...,n\}$, $\densz$ can be viewed as a generelisation of the space of probability distributions. See \cite{wilde} for a more detailed account of this statement. The $\rho$'s for which the diagonal of $D$ is a Dirac delta function are called pure states and their set is denoted $\densp$.\\

\subsubsection*{\textit{Fubini-Study metric}}
Note that a pure state $\rho \in \densp$ can also be written as follows: $\rho = \phi \phi^T$ where $\phi \in \mathbb{S}_{\mathbb{C}^n} \coloneqq \{\psi \in \mathbb{C}^n: ||\psi||_2 = 1\}$. The set of pure states can be identified with the complex projective space. This space is obtained by identifying two elements in $\mathbb{S}_{\mathbb{C}^n}$ that differ a complex factor, i.e. the quotient space $\mathbb{S}_{\mathbb{C}^n}/U(1)$. The identification between the pure states and $\mathbb{S}_{\mathbb{C}^n}/U(1) $ is given explicitly as follows: $\densp \ni \rho = \phi \phi^T \leftrightarrow [\phi] \in \mathbb{S}_{\mathbb{C}^n}/U(1)$. Combining the above, we get the following quotient map:
\begin{equation}
  \mathbb{S}_{\mathbb{C}^n} \ni \phi \mapsto \phi\phi^T \in \densp.
\end{equation}
If we equip the unit sphere with the Euclidean metric, the resulting quotient metric on the set of pure states will be the Fubini-Study metric with corresponding distance measure $d^{FS}([\phi],[\psi]) = |\langle \phi, \psi \rangle |$. See \cite{bengtsson} for a more comprehensive description.\\

From Section \ref{BWgeom} we know that the BW metric is obtained in a similar way. If we let $\mathbb{S}_{GL} = \{A \in \mathrm{GL}(n) : \tr(AA^*) = 1 \}$, the unit sphere in $\mathrm{GL}(n)$, then $\dens$ can be identified with $\mathbb{S}_{GL}/\mathrm{U}(n)$ such that $\rho \leftrightarrow [\rho^{1/2}]$. We have shown that when we equip $\mathbb{S}_{GL}$ with the Euclidean metric, the resulting quotient metric on $\dens$ is the BW metric. This shows that the BW metric and angle can be viewed as generalizations of the Fubini-Study metric and distance for mixed states.

\subsubsection*{\textit{Uhlmann and Takatsu}}

In 1992 the German theoretical physicist Armin Uhlmann gave a lecture at the Symposium of mathematical physics in Toruń, Poland titled \textit{"Density operators as an arena for geometry"} \cite{uhlmann}. In this talk he considers density operators as reductions of elements of a larger Hilbert space $\mathcal{H}^{\ext}$ called the purification space. If the reduction of a vector $M \in \mathcal{H}^{\ext}$ is equal to a density operator $\rho$ we call this vector a purification of $\rho$. This modus operandi is often referred to as: \textit{"Going to the church of the larger Hilbert space"} \cite{nilanjana}. $\mathcal{H}^{\ext}$ is given by the space of $n$-dimensional operators with the Hilbert-Schmidt (Euclidean) inner product and the reduction map is given by: $M \mapsto MM^* = \rho$. Note that the fiber of $\rho$ is given by: $\{\rho^{1/2}U: U \in \mathrm{U}(n)\}$. The fact that multiple vectors in $\mathcal{H}^{\ext}$ correspond to the same density operators is referred to as \textit{gauge freedom}. Minimizing the distance between purifications of $\rho_1$ and $\rho_2$ using this freedom gives the BW distance and computing the pushforward of the Hilbert-Schmidt inner product by the reduction map gives the BW metric. \\

As one can see, the approach taken by Uhlmann is similar as the argument in Section \ref{BWgeom}. The purification space takes the role of $\mathrm{GL}(n)$, the reduction map is the quotient map $\pi$ and the gauge freedom is justified by Theorem \ref{theorem1}. In 2008 an argument in a similar language to ours was given independently by Takatsu, this time for the Wasserstein distance between two mean-zero Gaussian distributions \cite{takatsu2008}. Since this distance is identical to the Bures distance, results from both fields can be carried over.

\subsubsection*{\textit{Wigner-Yanase information}}

In Section \ref{cig} it was described that the Fisher metric is the pushforward metric of the Euclidean metric under the square map. That is, the following map is an isometry:
\begin{equation}
  \left(\mathcal{M}_+(\Omega), \Euclh\right) \ni \mu \mapsto \mu^2 \in \left(\mathcal{M}_+(\Omega), \Fish\right).
\end{equation}
If one would replace $\MO$ by $\pherm$, the pushforward of $\Euclh$ will be the Wigner-Yanase metric as described in \cite{gibilisco}. \\

We will now describe what the metric on the left should be so that its pushforward becomes the BW metric. We write $(\cdot)_{\mathcal{H}(M)}: T_M \mathrm{GL}(n) \to \mathcal{H}(\pi, M, \Euclh)$ for the orthogonal projection on the horizontal subspace of $T_M \mathrm{GL}(n)$ where $\pi$ is still as defined in $\eqref{pi}$.
\begin{lemma}
    For $M \in \mathrm{GL}(n)$ and $A \in \mathcal{H}(\pi, M, \Euclh)$, we have:
    \begin{equation}
        \mathcal{L}_{MM^*}(d\pi_M A) = A M^{-1}.
    \end{equation}
\end{lemma}
\begin{proof}
    Since $A \in \mathcal{H}(\pi, M, \Euclh)$ we have that $AM^{-1} \in \herm$. We check:
    \begin{align}
        MM^* \left(AM^{-1}\right) + \left(A M^{-1}\right) M M^* &= MM^* \left(AM^{-1}\right)^* + \left(A M^{-1}\right) M M^*\\
        &= MA^* + A M^*\\
        &= d\pi_M A.
    \end{align} \qed
\end{proof}
We define the following metric on $\pherm$ for $H,K \in T_\Sigma \pherm$:
\begin{equation}
  \mathfrak{g}^\mathcal{H}_\Sigma(H,K) \coloneqq \re \tr \left( H_{\mathcal{H}(\Sigma)} \left(K_{\mathcal{H}(\Sigma)}\right)^* \right).
\end{equation}
Note that the projection $(\cdot)_{\mathcal{H}(\Sigma)}$ happens in the ambient space $T_\Sigma \mathrm{GL}(n)$ and therefore in general $H_{\mathcal{H}(\Sigma)},K_{\mathcal{H}(\Sigma)} \notin T_\Sigma \pherm$.
\begin{theorem}
  The pushforward metric of $\mathfrak{g}^\mathcal{H}$ under the square map is equal to the BW metric $\BWh$. That is, the following map is an isometry:
  \begin{align}
    \pi: \left(\pherm,\mathfrak{g}^\mathcal{H} \right) &\to \left(\pherm, \BWh\right) \\
    \Sigma &\mapsto \Sigma^2.
  \end{align}
\end{theorem}
\begin{proof}
  We have:
  \begin{align}
      \BWh_{\Sigma^2}(d\pi_\Sigma H, d\pi_\Sigma K) &= \BWh_{\Sigma^2}(d\pi_\Sigma H_{\mathcal{H}(\Sigma)}, d\pi_\Sigma K_{\mathcal{H}(\Sigma)})\\
      &= \frac{1}{2} \re \tr \left( \mathcal{L}_{\Sigma^2}(d\pi_\Sigma H_{\mathcal{H}(\Sigma)}) d\pi_\Sigma K_{\mathcal{H}(\Sigma)} \right)\\
      &= \frac{1}{2} \re \tr\left( H_{\mathcal{H}(\Sigma)} \Sigma^{-1} \left(\Sigma (K_{\mathcal{H}(\Sigma)})^* + K_{\mathcal{H}(\Sigma)} \Sigma \right) \right)\\
      &= \re \tr \left(H_{\mathcal{H}(\Sigma)} (K_{\mathcal{H}(\Sigma)})^*\right)\\
      &= \mathfrak{g}^\mathcal{H}_\Sigma(H,K).
  \end{align} \qed
\end{proof}

\section*{Conclusion}

 In this paper we have highlighted the interplay between Fisher theory and quantum information theory. This has been done by  presenting an existing derivation of the Riemannian metric corresponding to the BW distance on $\pherm$ and subsequently adapting this argument so that the Riemannian distance and geodesics could be  recovered for this metric on the subset $\dens$. In the last part we have compared the geometrical structure to similar structures within quantum information.

\subsubsection*{\textit{Further questions}}
The geometrical structure derived in this paper lives on the space of positive-definite matrices. In quantum information, the larger set of positive semi-definite matrices is studied. It would be interesting to investigate whether the argument to obtain this geometrical structure can be generalised to this set. \\

The Talagrand inequality in its original form \cite{talagrand} gives a bound on the 2-Wasserstein distance between two Gaussian distributions in terms of their relative entropy:
\begin{equation}
  W^2(\mu,\nu) \leq D(\mu||\nu).
\end{equation}
The Wasserstein distance can be written in terms of the covariance matrices of its arguments. The current exposition shows that written in this form, this distance measure appears in quantum information. A related distance measure on the space of covariance matrices is the quantum relative entropy. A further study of the relation between these two distances could result in a distance measure between two (classical) distributions in terms of the quantum relative entropy of their covariance matrices, similar to \cite{georgiou}, and could potentially lead to a Talagrand-type inequality.

% \vspace*{300px}

\section*{List of symbols} \label{notation}
\begin{align*}
  &\Omega: &&\text{Sample space}\\
  & &&\{\omega_1, ..., \omega_n\} \\    
  &\mathcal{S}(\Omega): &&\text{Space of signed measures}  \\
  & &&\{\mu: \Omega \to \mathbb{R} \} \\
  &\mathcal{M}_+(\Omega): &&\text{Space of strictly positive measures}\\
  & &&\{\mu \in \mathcal{S}(\Omega): \mu(\{\omega_i\}) > 0, \  \forall i \in \{1,...,n\}\} \\
  &\mathcal{P}_+(\Omega) &&\text{Space of strictly positive probability measures}\\
  & &&\{p \in \mathcal{M}_+(\Omega): \sum_i p(\omega_i) = 1\} \\
  &\mathrm{GL}(n): &&\text{General Linear group}\\
  & &&\{M \in \mathbb{C}^{n \times n} : \det(M) \neq 0\} \\
  &\mathrm{U}(n): &&\text{Unitary matrices}  \\
  & &&\{ U \in \mathrm{GL}(n): U^*U = I \} \\
  &\herm: &&\text{Hermitian matrices} \\
  & &&\{H \in \mathbb{C}^{n \times n}: H = H^*\} \\
  &\mathbb{H}^{\perp}(n): &&\text{Skew-Hermitian matrices} \\
  & &&\{K \in \mathbb{C}^{n \times n}: K = -K^*\} \\
  &\pherm: &&\text{Positive-definite Hermitian matrices} \\
  & &&\{\Sigma \in \herm: \phi^* \Sigma \phi > 0, \forall \phi \in \mathbb{C}^n \} \\
  &\dens: &&\text{Trace-one positive-definite Hermitian matrices} \\
  & &&\{\rho \in \pherm : \tr(\rho) = 1\}\\
  &\mathbb{S}_{\mathbb{C}^{n\times n}}:  &&\text{Unit sphere in matrix space} \\
  & &&\{A \in \mathbb{C}^{n\times n} : \tr(AA^*) = 1\} \\
  &\sgl: &&\text{Unit sphere in invertible matrix space} \\
  & &&\mathbb{S}_{\mathbb{C}^{n\times n}} \cap \mathrm{GL}(n)\\
  &\mathcal{L}_\Sigma(H): &&\text{Lyapunov operator} \\
  & &&X \text{ such that } X\Sigma + \Sigma X = H\\
  &d_\pherm^{BW}(\Sigma,T): &&\text{Bures-Wasserstein distance} \\
  & &&\left[\tr(\Sigma) + \tr(T) - 2 \tr\left(\left(\Sigma^{1/2}T\Sigma^{1/2}\right)^{1/2}\right)\right]^{1/2}\\
  & d^{BW}_\dens(\rho_1, \rho_2) && \text{Bures-Wasserstein angle} \\
  & &&\arccos\left(\tr\left(\left(\rho_2^{1/2}\rho_1\rho_2^{1/2}\right)^{1/2}\right)\right)\\
  &\BW{H}{K}{\Sigma}: &&\text{Bures-Wasserstein metric} \\
  & &&\frac{1}{2} \re \tr\left(\mathcal{L}_\Sigma\left(H^{(m)}\right) K^{(m)}\right)\\
  &L(\gamma): &&\text{Length of a curve} \\
  & &&\int_a^b ||\gamma'(t)||_g dt\\
  &d_\mathcal{M}(p,q): &&\text{Riemannian distance on a manifold $\cM$} \\
  & &&\inf \{ L(\gamma) \mid \gamma: [0,1] \to \mathcal{M}, \gamma(0) = p, \gamma(1) = q \}\\
  &\Euclh: && \text{Euclidean metric} \\
  &\bar{d}: &&\text{Euclidean distance} \\
  &d'([p],[q]): && \text{Quotient distance}\\
  & &&\inf_{p \in [p], q \in [q]} d(p,q)  
\end{align*}

\section*{Data Availability}
Data sharing not applicable to this article as no datasets were generated or analysed during the current study.

\bibliographystyle{plain}
\bibliography{refs}

\end{document}